\documentclass{article}
\usepackage[utf8]{inputenc}
\usepackage{amsmath}
\usepackage{amsthm}
\usepackage{amssymb}
\usepackage{tikz}
\usetikzlibrary{arrows, automata}
\usetikzlibrary{intersections}
\usepackage{hyperref}
\usepackage{cite}
\usepackage{authblk}
\usepackage{graphicx}
\usepackage{adjustbox}
\usepackage{caption}
\usepackage{subcaption}
\usepackage{appendix}
\usepackage{pgfplots}
\usepgfplotslibrary{fillbetween}
\pgfplotsset{compat=1.11}
\pgfdeclarelayer{bg}
\pgfsetlayers{bg,main}

\newtheorem{result}{Result}

\newtheorem{obs}{Observation}
\newtheorem{theorem}{Theorem}
\newtheorem{lemma}{Lemma}
\newtheorem{corollary}{Corollary}
\newtheorem{prop}{Proposition}
\newtheorem{problem}{Problem}

\providecommand{\subjclass}[1]{\textbf{2010 AMS Subject Class:} #1}
\providecommand{\keywords}[1]{\textbf{Keywords:} #1}
\title{Dominator Colorings of Digraphs}
\author[1]{Michael Cary\footnote{macary@mix.wvu.edu}}
\affil[1]{West Virginia University}
\begin{document}
\maketitle
\begin{abstract}
This paper serves as the first extension of the topic of dominator colorings of graphs to the setting of digraphs. We establish the dominator chromatic number over all possible orientations of paths and cycles. In this endeavor we discover that there are infinitely many counterexamples of a graph and subgraph pair for which the subgraph has a \emph{larger} dominator chromatic number than the larger graph into which it embeds. Finally, a new graph invariant measuring the difference between the dominator chromatic number of a graph and the chromatic number of that graph is established and studied. The paper concludes with some of the possible avenues for extending this line of research.
\end{abstract}

\subjclass{05C69, 05C20, 05C15}

\keywords{dominator coloring, digraph, domination}

\section{Introduction}\label{s1}
Dominator colorings of graphs are a variant of the longstanding problem of finding a proper coloring of the vertex set of a graph. The topic of dominator colorings can be traced back to the work of Gera in \cite{gera2007}. A proper dominator coloring of a graph $G$ is a coloring $\mathcal{C}$ of the vertices $V(G)$ such that $\mathcal{C}$ is a proper vertex coloring and every vertex $v\in V(G)$ dominates some color class in $\mathcal{C}$. This type of graph coloring problem helps to relate other problems involving domination and related topics in graph theory \cite{harary1996,haynes1998}.

Initial results in this relatively new area have been bountiful, ranging from general results to tighter results on special classes of graphs \cite{abid2019,alikhani2015,arumugam2012,gera2006,gera2007bi,kavitha2012,merouane2012,panda2015}. Results on dominator colorings of graphs that are products of elementary graphs, such as paths and cycles, as well as other graph operations, were studied in \cite{chen2017} and \cite{paulraja2016}.

Extending this foundational line of research, \cite{baogen2000} made headway in finding extremal graphs with respect to various domination parameters. Algorithmic aspects were studied in \cite{arumugam2011} in which it was determined that even for several rather elementary classes of graphs, dominator colorings cannot be found in polynomial time. This result answered a standing question posed in \cite{chellali2012}. Further results on computation complexity of pertinent algorithms were developed in \cite{shalu2017}, and improved algorithms for special classes of graphs, such as trees were given in \cite{merouane2015}. In addition, These results have proven useful for the development of many domination based network theoretic tools and applications, including those found in \cite{blair2011,davis2016,desormeaux2018,haynes2002,haynes2003,somasundaram1998}.

Broadly, this field has blossomed into a rather diverse collection of variants of domination and independence, in some cases combined with graph coloring problems. Examples include total domination \cite{henning2015,kazemi2013,kazemi2015}, power domination \cite{chang2012,dorbec2008}, broadcast domination \cite{brewster2013,dunbar2006,erwin2004,mynhardt2017,teshima2012}, and geodetic domination \cite{escuadro2011}, among others. In addition to these highly similar problems, more nuanced relations between domination and other areas of graph theory exist, such as decycling or network dismantling problems \cite{bau2007,beineke1997,cary2018,yang2019}.

The direction this paper takes is to initiate the extension of the study of dominator colorings of graphs to the natural setting of digraphs. By definition digraphs have notions of domination embedded into the foundation of their own existence; the definition of an arc as an ordered pair, rather than an unordered pair, of vertices speaks strongly to this idea. While various forms of vertex coloring problems exist in digraphs beyond the standard problem, see, e.g., \cite{neumann1982}, perhaps none is quite as suited for digraphs as dominator coloring.

Formally, a dominator coloring of a digraph $D$ is a vertex coloring $\mathcal{C}$ of the vertex set $V(D)$ such that $\mathcal{C}$ is a proper vertex coloring and for all $v\in V(D)$ there exists some $C\in\mathcal{C}$ such that for all $u\in C$ the arc $vu$ exists as a member of the arc set $A(D)$. As will be seen throughout this paper, even in the most elementary of settings this problem becomes quite tedious rather quickly.

Before beginning this study, we take the time to elucidate several highly used notations and assumptions. All digraphs are simple, loopless, and connected unless specified otherwise. The notation $G_{D}$ of underling graph of a digraph $D$. A directed path $P_{n}=v_{1}v_{2}\dots v_{n}$ is the orientation of a path of length $n$ whose out-degree sequence is given by $\{1,1,1,\dots,1,1,0\}$. A directed cycle $C_{n}=v_{1}v_{2}\dots v_{n}v_{1}$ is the orientation of a cycle of length $n$ whose out-degree sequences is given by $\{1,1,1,\dots,1,1,1\}$. 

To maintain traditional (dominator) chromatic number notation, we use the notation $\chi_{d}(D)$ throughout this text. It is very important to note now that the value $\chi_{d}(D)$ is in reference to the cardinality of a smallest possible set of color classes used in any proper dominator coloring of any possible orientation of the digraph $D$. That is to say, this paper focuses in particular with the problem of finding the smallest possible (minimum) dominator coloring, i.e., the dominator chromatic number, over all possible orientations of a graph. The problem of finding the dominator chromatic number of specific orientations of digraphs is not addressed in this work.

\section{Preliminary Results}\label{s2}
In this section we provide some preliminary results on elementary digraph structures and dominator colorings of digraphs. We begin with two trivial observations and a result that shows the incredible nature of dominator colorings of digraphs.
\begin{obs}\label{o1}
For any digraph $D$, we have that $\chi_{d}(D)\geq\chi(D)$.
\end{obs}
\begin{obs}\label{o2}
For any digraph $D$, we have that $\chi_{d}(D)\leq|V(D)|$.
\end{obs}
\begin{lemma}\label{l1}
It is \textbf{NOT} true that for any digraph $D$ and sub-digraph $H\subset D$, $\chi_{d}(H)\leq\chi_{d}(D)$.
\end{lemma}
\begin{proof}
Consider the path $P_{4}$ with out-degree sequence $\{0,2,0,1\}$ and the cycle $C_{4}$ with out-degree sequence $\{0,2,0,2\}$. As will be shown later, these are orientations admitting minimum dominator colorings of each digraph. However, $\chi_{d}(P_{4})=3>2=\chi_{d}(C_{4})$.
\end{proof}

This result is what makes dominator colorings of digraphs stand out from all other variants of vertex coloring. However, using the fact that $\chi_{d}(P_{n})\leq\chi_{d}(C_{n})$ for $m>4$ will be very important later on in this work, and in fact will be proven in Section \ref{s3} in between proving the minimum dominator chromatic number of paths and cycles. However, as it turns out, this is not a unique instance. We present another, similar result showing that there are infinitely many counterexamples to the claim that for every digraph $D$ and subdigraph $H$, $\chi_{d}(H)\leq\chi_{d}(D)$.

It is important to emphasize again that throughout this paper we are concerned with finding the smallest possible dominator coloring over all orientations of a particular graph structure.  It is easy to see that we may reduce the number of colors used in a minimum proper coloring of a given orientation of a digraph by embedding this digraph into a larger digraph. In addition to the pair $(C_{4},P_{4})$, it turns out that there are indeed other examples of this phenomenon. With the goal of constructing an infinite family of digraphs satisfying this relationship ($\chi_{d}(H)>\chi_{d}(D)$ for a digraph $D$ and a subdigraph $H$ of $D$), consider the following example (which spans two figures). 

\begin{figure}[h!]
\centering
\begin{tikzpicture}[-,>=stealth',shorten >=1pt,auto,node distance=2cm,
                    thick,main node/.style={circle,draw}]
  \node[main node] (A)					{$v_{1}$};
  \node[main node] (B) [right of=A]		{$v_{2}$};
  \node[main node] (C) [right of=B] 	{$v_{3}$};
  \node[main node] (D) [below of=C] 	{$v_{4}$};
  \node[main node] (E) [left of=D]      {$v_{5}$};
  \node[main node] (F) [left of=E]      {$v_{6}$};  
  \draw[thick,->] (A) to (B);
  \draw[thick,->] (B) to (C);
  \draw[thick,->] (C) to (D);
  \draw[thick,->] (D) to (E);
  \draw[thick,->] (E) to (F);
  \draw[thick,->] (F) to (A);
\end{tikzpicture}
\caption{This digraph $H$ is an orientation of $C_{6}$ which requires $6$ colors in a minimum proper dominator coloring.}
\label{f1}
\end{figure}
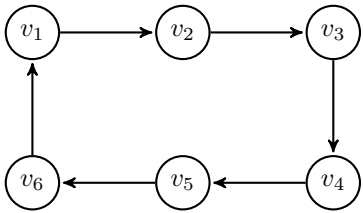
\begin{figure}[h!]
\centering
\begin{tikzpicture}[-,>=stealth',shorten >=1pt,auto,node distance=2cm,
                    thick,main node/.style={circle,draw}]
  \node[main node] (A)					{$v_{0}$};
  \node[main node] (B) [above of=A]		{$v_{1}$};
  \node[main node] (C) [left of=B]  	{$v_{2}$};
  \node[main node] (D) [right of=B] 	{$v_{3}$};
  \node[main node] (E) [below of=A]     {$v_{4}$};
  \node[main node] (F) [left of=E]      {$v_{5}$};  
  \node[main node] (G) [right of=E]     {$v_{6}$};  
  \draw[thick,->] (B) to (A);
  \draw[thick,->] (C) to (A);
  \draw[thick,->] (D) to (A);
  \draw[thick,->] (E) to (A);
  \draw[thick,->] (F) to (A);
  \draw[thick,->] (G) to (A);
  \draw[thick,->] (C) to (B);
  \draw[thick,->] (B) to (D);
  \draw[thick,->] (D) to (G);
  \draw[thick,->] (G) to (E);
  \draw[thick,->] (E) to (F);
  \draw[thick,->] (F) to (C);
\end{tikzpicture}
\caption{A digraph $D$ satisfying $H\subset D$ and $\chi_{d}(D)=3<6=\chi_{d}(H)$.}
\label{f2}
\end{figure}
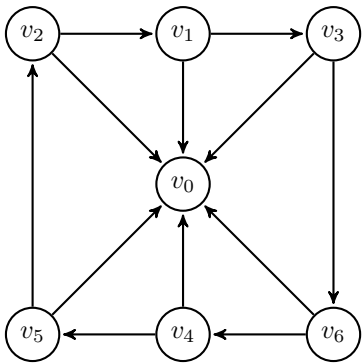

While the formal results for the dominator chromatic number of a cycle will be proven later in this paper (see Theorem \ref{t2}), we claim now that it is larger than the chromatic number of the underlying cycle, which is either two or three depending on the parity of the size of the cycle. By adding a single vertex that is adjacent to every vertex in the cycle, we may orient this graph such that the new vertex is a sink and the remaining vertices are oriented into a directed cycle, thereby allowing us to color the cycle on either two or three colors, depending on parity, and then using only one more color for the sink. Call such a graph (including the sink) $\tilde{C_{n}}$. The family of digraphs $\{\tilde{C_{n}}\}_{n=7}^{\infty}$ constitutes infinitely many counterexamples to the claim that for any digraph $D$ and subdigraph $H$, $\chi_{d}(H)\leq\chi_{d}(D)$. The reason that the index starts at seven is a consequence of the dominator chromatic number of cycles and comes from the result to be proven in Theorem \ref{t2}.

To formalize this observation as a problem in the topic of dominator colorings of digraphs, we introduce the following notation. Let $H\subset D$ be a sub-digraph of a digraph $D$ with $H$ and $D$ satisfying $\chi_{d}(H)>\chi_{d}(D)$. Let $\delta(D,H)=\chi_{d}(H)-\chi_{d}(D)$ denote the \textit{dominator discrepancy} of $H$ in $D$. A very interesting problem not addressed in this paper would be to find: (1) which digraphs have positive dominator discrepancies; (2) what the largest dominator discrepancies are for various families of digraphs; (3) which sub-digraphs are responsible for the largest dominator discrepancies in given families of digraphs; and (4) do some particular families of digraphs and particular families of sub-digraphs of these digraphs offer well-parameterized dominator discrepancies?

We proceed by studying perhaps the two most fundamental structures, paths and cycles, in their directed setting as a basis for the main results to come in Section \ref{s3}. While this paper is concerned with smallest possible dominator colorings, we consider now the natural questions of finding the dominator chromatic number of directed paths and cycles. In the process, we find an orientation which required the largest possible number of colors for a minimum proper dominator coloring over all orientations of paths and cycles, namely, the directed path and the directed cycle.

\begin{prop}\label{p1}
For any directed path $P$ of order $n$, we have that $\chi_{d}(P)=n$.
\end{prop}
\begin{proof}
The proof is by induction on the length of the directed path. As our basis is obvious, assume that for every directed path of length $k<n$ we have that $\chi_{d}(P)=k$ and consider a directed path of length $n$ given by $P=v_{0}v_{1}\dots v_{n-1}$. If we remove the vertex $v_{0}$ we obtain a directed path $P^{\prime}$ of length $n-1$ which, by our inductive hypothesis, has $\chi_{d}(P^{\prime})=n-1$. We immediately see that $c(v_{0})\neq c(v_{1})$ must hold, else we do not have a proper coloring of $P$. If we chose to color $v_{0}$ so that $c(v_{0})=c(v_{i})$ for some $i\in\{2,\dots, n-1\}$ then the vertex $v_{i-1}$ would no longer dominate a color class. Therefore we must assign a new color to $v_{0}$ and obtain that $\chi_{d}(P)=n$.
\end{proof}
\begin{prop}\label{p2}
For any directed cycle $C$ or order $n$, we have that $\chi_{d}(C)=n$.
\end{prop}
\begin{proof}
Let $C=v_{0}v_{1}\dots v_{n-1}v_{0}$ be a directed cycle of order $n$ and assume that $\mathcal{F}$ is a total dominator coloring of $C$ attaining $\chi_{d}(C)=m<n$. Then there is some color class $c_{i}\in\mathcal{F}$ such that $|c_{i}|\geq 2$. Without loss of generality, let $v_{i}$ and $v_{j}$ be two vertices with color $c_{i}$ in our coloring of $C$, and consider the arc $v_{i-1}v_{i}$ of $C$. Clearly $v_{i-1}$ does not dominate any color class of $\mathcal{F}$, a contradiction. Therefore it must be that $\chi_{d}(C)=n$.
\end{proof}

At this point, it would seem natural to determine if the existence of a Hamiltonian directed path/cycle in a digraph $D$ is a sufficient condition for the digraph to have $\chi_{d}(D)=|V(D)|$. As the following two examples demonstrate, these are certainly not sufficient conditions for a digraph $D$ to satisfy $\chi_{d}(D)=|V(D)|$.

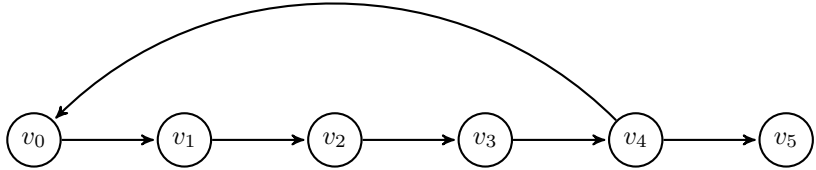
\begin{figure}[h!]
\centering
\begin{tikzpicture}[-,>=stealth',shorten >=1pt,auto,node distance=2cm,
                    thick,main node/.style={circle,draw}]
  \node[main node] (A)					{$v_{0}$};
  \node[main node] (B) [right of=A]		{$v_{1}$};
  \node[main node] (C) [right of=B] 	{$v_{2}$};
  \node[main node] (D) [right of=C] 	{$v_{3}$};
  \node[main node] (E) [right of=D]     {$v_{4}$};
  \node[main node] (F) [right of=E]     {$v_{5}$};  
  \draw[thick,->] (A) to (B);
  \draw[thick,->] (B) to (C);
  \draw[thick,->] (C) to (D);
  \draw[thick,->] (D) to (E);
  \draw[thick,->] (E) to (F);
  \draw[thick,->] (E) to [out=135,in=45] (A);
\end{tikzpicture}
\caption{An example of a digraph with a Hamiltonian directed path that has $\chi_{d}(D)<|V(D)|$. We may assign the vertices $v_{0}$ and $v_{5}$ to the same color class even though the digraph has a Hamiltonian directed path $v_{0}v_{1}v_{2}v_{3}v_{4}v_{5}$.}
\label{f3}
\end{figure}
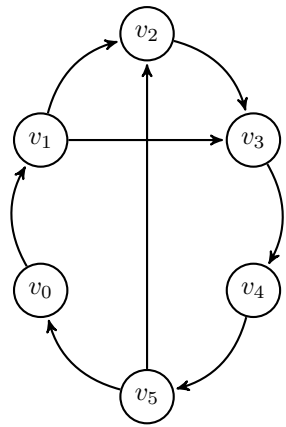
\begin{figure}[h!]
\centering
\begin{tikzpicture}[-,>=stealth',shorten >=1pt,auto,node distance=2cm,
                    thick,main node/.style={circle,draw}]
  \node[main node] (A)					      {$v_{0}$};
  \node[main node] (B) [above of=A]		      {$v_{1}$};
  \node[main node] (C) [above right of=B] 	  {$v_{2}$};
  \node[main node] (D) [below right of=C] 	  {$v_{3}$};
  \node[main node] (E) [below of=D]           {$v_{4}$};
  \node[main node] (F) [below left of=E]      {$v_{5}$};  
  \draw[thick,->] (A) to [bend left] (B);
  \draw[thick,->] (B) to [bend left] (C);
  \draw[thick,->] (C) to [bend left] (D);
  \draw[thick,->] (D) to [bend left] (E);
  \draw[thick,->] (E) to [bend left] (F);
  \draw[thick,->] (F) to [bend left] (A);
  \draw[thick,->] (F) to (C);
  \draw[thick,->] (B) to (D);
\end{tikzpicture}
\caption{An example of a digraph with a Hamiltonian directed cycle that has $\chi_{d}(D)<|V(D)|$. We may assign the vertices $v_{0}$ and $v_{2}$ to the same color class even though the digraph has a Hamiltonian directed cycle $v_{0}v_{1}v_{2}v_{3}v_{4}v_{5}v_{0}$.}
\label{f4}
\end{figure}

To conclude this section, we provide an analogous proposition for the dominator chromatic number of orientations of star graphs. A star graph is simply the complete bipartite graph $K_{1,k}$. As the following lemma will show, star graphs are perhaps as elementary of a class of graphs as possible in which the dominator chromatic number $\chi_{d}(D)$ is not invariant under orientation.
\begin{prop}\label{p3}
Let $D$ be an orientation of a star graph, $G$. Then $2\leq\chi_{d}(D)\leq3$, $\chi_{d}(D)=2$ if and only if all arcs are oriented similarly with respect to the central vertex, and $\chi_{d}(D)=3$ otherwise.
\end{prop}
\begin{proof}
First assume that all arcs of $D$ are oriented similarly with respect to the central vertex, call it $v$. Clearly we may color $D$ by assigning a unique color to $v$ and a common color (distinct from the color $c(v)$) to $V(D)\setminus\{v\}$.

Next, assume that all arcs of $D$ do not share a similar orientation with respect to $v$. Again color $v$ with a unique color $c(v)$. Then assign one color to each of $N^{+}(v)$ and $N^{-}(v)$. This establishes a dominator coloring of $D$ using three colors, so in order to complete the entire proof, it remains to be shown that this coloring uses the fewest possible colors for this orientation. Clearly we must have distinct colors for $N^{+}(v)$ and $N^{-}(v)$, else $v$ does not dominate any color class. Similarly we have that $v$ and each of $N^{-}(v)$ and $N^{+}(v)$ must not share any colors. Thus, under such an orientation, $\chi_{d}(D)=3$ and the entire proof is complete.
\end{proof}

Thus we see that the dominator chromatic number of a star graph directly indicates the uniformity (or lack thereof) in arc orientation. Equivalently, for an orientation $D$ of a star graph, $\chi_{d}(D)=2$ if and only if either $D$ or $D^{-}$, the digraph obtained by reversing the orientation of every arc of $D$, results in an arborescence. Clearly the dominator chromatic number of a digraph can be directly indicative of crucial structural properties of the digraph. Characterizing to what extent this is the case would prove an insightful contribution.

\section{Orientations of Paths}\label{s3}
In the previous section we determined that directed paths, directed cycles, and tournaments all obtain a largest possible dominator chromatic number. Given this, at least in the cases of paths and cycles, we might ask whether or not this particular orientation of a path or cycle is unique in maximizing the dominator chromatic number of the digraph. We begin by studying orientations of paths, then proceed to orientations of cycles, and then further generalize orientations of paths by looking at specific types of orientations of trees, chiefly orientations of stars.

In an arbitrary orientation of a path, the vertices may have out degree between zero and two, with there being precisely one more out degree zero vertex than out degree two vertex (due to the degree sum formula). We begin by proving the minimum dominator chromatic number over all possible orientations of paths. Before doing this, however, we present a lemma about dominator coloring of sub-paths of paths dealing with containment.
\begin{lemma}\label{l2}
Let $m,n\in\mathbb{N}$ such that $m<n$ and let $P_{m}$ and $P_{n}$ be orientations of paths of length $m$ and $n$, respectively. If $P_{m}\subset P_{n}$ then $\chi_{d}(P_{m})\leq\chi_{d}(P_{n})$.
\end{lemma}
\begin{proof}
It is obvious that no color classes of $P_{m}$ may be combined if $P_{m}$ is embedded into a larger path when $P_{m}$ has been given a proper dominator coloring of the fewest possible colors, given its orientation.
\end{proof}

\begin{theorem}\label{t1}
The minimum dominator chromatic number over all orientations of the path $P_{n}$ is given by
\begin{equation*}
\chi_{d}(P_{n})=\begin{cases}
k+2 & \mathrm{if}\ n=4k\\
k+2 & \mathrm{if}\ n=4k+1\\
k+3 & \mathrm{if}\ n=4k+2\\
k+3 & \mathrm{if}\ n=4k+3
\end{cases}
\end{equation*}
for $k\geq 1$ with the exception $\chi_{d}(P_{6})=3$.
\end{theorem}
\begin{proof}
Our proof consists of proving a series of claims which, collectively, prove the theorem. For clarity, the completion of the proof of each claim is denoted by a diamond symbol, i.e., by $\Diamond$. Additionally, notice that the dominator chromatic number of $P_{n}$ for $n<4$ is given by
\begin{equation*}
\chi_{d}(P_{n})=\begin{cases}
1 & \mathrm{if}\ n=1\\
2 & \mathrm{if}\ n=2\\
2 & \mathrm{if}\ n=3
\end{cases}
\end{equation*}
We begin this proof with an obvious but important claim.

\textbf{Claim 1.} Let $S=\{v\in V(P_{n})|d^{+}(v)=2\}$. Then $S$ is colored completely by one color class in any minimum dominator coloring of $P_{n}$.

\textbf{Proof.} No vertex in a path can dominate a vertex of out-degree two. $\Diamond$

Next, we show that a particular structure yields a minimum dominator coloring over all orientations of the path $P_{n}$ for $n=4k+1$ where $n$ is a positive integer. Unless otherwise specified, we assume that $k\geq1$.

\textbf{Claim 2.} Let $P_{n}=v_{0}v_{1}\dots v_{n-1}$ be an orientation of a path on $n$ vertices and let $\mathcal{C}$ be a dominator coloring of $P_{n}$ using fewest possible colors. Then we have that $\nexists\ v_{i},v_{i+1}\in V(P_{n})\ \mathrm{s.t.}\ d^{+}(v_{i})=d^{+}(v_{i+1})=1$.

\textbf{Proof.} Define $P_{n}=v_{0}v_{1}\dots v_{n-1}$. There are two cases to consider; either $d^{+}(v_{0})=d^{+}(v_{1})=1$ or there exists a subsequence of $P_{n}$, say $v_{i-1},v_{i},v_{i+1}$ with out-degree sequence $\{2,1,1\}$. 

In the first case, we have that the vertex $v_{0}$ must be colored with the same color as the vertices of out-degree two, i.e., the color assigned to all vertices not dominated by any other vertex, else $\mathcal{C}$ is not minimal. However, since $d^{+}(v_{0})=1$, the color assigned to $v_{1}$ must be unique. Similarly, since $d^{+}(v_{1})=1$, the color assigned to $v_{2}$ must be unique. This constitutes three color classes used on the vertices $v_{0}$, $v_{1}$, and $v_{2}$, with two of these colors being unique. Re orient the arc $v_{0}v_{1}$ so that it becomes $v_{1}v_{0}$. If we leave the color assigned to $v_{2}$, recolor $v_{0}$ with the same color as $v_{2}$, and color $v_{1}$ with the color assigned to all non-dominated vertices, we recolor $P_{n}$ with fewer colors, contradicting that $\mathcal{C}$ was a minimum dominator coloring.

For the second case, consider the vertex sequence $v_{i-1},v_{i},v_{i+1}$ with out-degree sequence $\{2,1,1\}$. By Claim 1, we have that $v_{i-1}$ is colored with the same color assigned to all non-dominated vertices in $P_{n}$, call it $c_{2}$. By extending the vertex sequence by one vertex in each direction, we see that we have the subpath $v_{i-2}v_{i-1}v_{i}v_{i+1}v_{i+2}$ assigned the colors $c_{x},c_{2},c_{y},u_{1},u_{2}$ where it may be the case that $c_{x}=c_{y}$ and the colors $u_{i}$ are unique, i.e., they are only assigned to one vertex in all of $P_{n}$. By reversing the orientation of the arc $v_+{i}v_{i+1}$, we get that $d^{+}(v_{i})=0$ and $d^{+}(v_{i+1})=2$. We may recolor this subpath with the colors $c_{x},c_{2},c_{y},c_{2},u_{2}$ to achieve a proper dominator coloring of $P_{n}$ with fewer colors than $\mathcal{C}$, contradicting that $\mathcal{C}$ was a minimum dominator coloring of $P_{n}$. Therefore we may conclude that in a minimum dominator coloring over all orientations of the path $P_{n}$, there are no consecutive vertices with out-degree one. $\Diamond$

As a consequence, we have that a valid out-degree sequence for an orientation of a path on $n=4k+1$ vertices, for any $k\in\mathbb{N}$, is $\{0,2,0,2,0,\dots,2,0,2,0\}$. The next claim shows that this is necessarily the optimal structure.

\textbf{Claim 3.} For $k>2$ there are no vertices with out-degree equal to one in an orientation of $P_{4k+1}$ which admits a minimum dominator coloring.

\textbf{Proof.} Let $P_{4k+1}$ be an orientation with vertices of out-degree one. Per Claim 2, no two consecutive vertices of out-degree one exist. Notice that there must be an even number of vertices with out-degree one, as each vertex of out-degree one is between a vertex of out-degree two and a vertex of out-degree zero, and the removal of all vertices of out-degree one would leave a path whose out-degree sequence is $\{0,2,0\dots,0,2,0\}$ which necessarily has an odd number of members. Thus we may conclude that there are at least two vertices with out-degree one. Notice next that each vertex dominated by a vertex of out-degree one must be uniquely colored. Thus, if we remove all vertices of out-degree one from $P_{4k+1}$, the resulting path has a proper dominator coloring. Consider the vertices of out-degree two which dominated vertices of out-degree one that were not uniquely colored in this minimal dominator coloring of $P_{4k+1}$. The vertex of out-degree zero that was dominated by the vertex of out-degree two in the original path was either uniquely colored, in which case the vertex of out-degree two now dominates two color classes, or the vertex of out-degree zero shared a color with the vertex of out-degree one that we removed. In either case, the vertex of out-degree two dominates two color classes in the resultant path. Let $P^{\prime}$ be the subpath of $P_{4k+1}$ obtained by removing all vertices of out-degree one, and let $\mathcal{C}$ and $\mathcal{C}^{\prime}$ be their respective minimal dominator colorings. Let there have been $2m$ vertices of out-degree one in $P_{4k+1}$ (recall that there are necessarily an even number of such vertices). By coloring the neighbors of each vertex of out-degree two that originally dominated a vertex of out-degree one with the same color, we have that $\mathcal{C}^{\prime}$ uses $2m-1$ fewer colors than $\mathcal{C}$. We can easily append the $2m$ removed vertices to either end of $P^{\prime}$ such that the out-degree sequence of the new path on $4k+1$ vertices is $\{0,2,0,\dots,0,2,0\}$ using fewer than $2m-1$ new colors simply by coloring all new vertices of out-degree two with the same color used to color all existing out-degree two vertices.

All that remains to be shown, then, is the case where $m=1$, i.e., when there were exactly two vertices of out-degree one in $P_{4k+1}$. One can construct examples which obtain a minimum dominator coloring using the out-degree sequence $\{1,0,2,0,1\}$ and $\{1,0,2,0,2,0,2,0,1\}$ for $P_{5}$ and $P_{9}$, respectively. However, when $k>2$, if we assume the same out-degree sequence, we can remove the outermost two vertices from each end ($\{v_{1},v_{2},v_{4k},v_{4k+1}$) and obtain a smaller path with the same out-degree sequence which, presumably, also admits a minimum dominator coloring. However, if this is the case, it is easy to see that to extend this dominator coloring to the original path $P_{4k+1}$, we need two more colors. If this was not the case, then $P_{4(k-1)+1}$ would not have been given a minimum dominator coloring as we had assumed. The extension from $P_{5}$ to $P_{9}$ exists only because the vertex $v_{5}$ dominates all out-degree zero vertices in $P_{9}$ not dominated by a vertex of out-degree one. This is not possible for $k>2$. 

Therefore we conclude that any minimum dominator coloring of the paths $P_{4k+1}$ has no vertices with out-degree one.$\Diamond$

Thus we have that the unique out-degree sequence of the path $P_{4k+1}$ which admits a minimum dominator coloring over all orientations of $P_{4k+1}$ for $k>2$.

\textbf{Claim 4.} Let $n=4k+1$ for some $k\in\mathbb{N}$. Then $\chi_{d}(P_{n})\geq k+2$.

\textbf{Proof.} From Claim 2 we know that the out-degree sequence of $P_{n}$ is precisely $\{0,2,0,\dots,0,2,0\}$. We must color all vertices with out-degree two with the same color. Moreover, each vertex with out-degree two must dominate some color class. Since there are $2k$ vertices with out degree two, we must introduce at least $\frac{2k}{2}=k$ more color classes; we may reduce this from $2k$ to $k$ on the basis that it is possible to have as many as two vertices with out-degree two dominate the same color class, but never more than two vertices may dominate the same color class in an orientation of a path. Then, if the remaining vertices are all colored with the same, new color (these vertices are dominated), we have used $k+2$ colors in total. $\Diamond$

\textbf{Claim 5.} $\chi_{d}(P_{4k+1})=k+2$.

\textbf{Proof.} Orient $P_{4k+1}$ such that the out-degree sequence is given by the sequence $\{0,2,0,\dots,0,2,0\}$. From Claim 3, it suffices to show that this orientation of $P_{n}$ for $n=4k+1$ uses only $k+2$ colors to give a proper dominator coloring of $P_{n}$. Let $P_{n}=v_{1}v_{2}\dots v_{4k+1}$. Color $P_{n}$ with the function $c$ given below.
\begin{equation*}
c(v_{i})=\begin{cases}
c_{1} & \mathrm{if}\ i\equiv\ 0\ (\mathrm{mod}\ 4)\\
c_{2} & \mathrm{if}\ i\equiv\ 1\ (\mathrm{mod}\ 2)\\
\mathrm{unique} & \mathrm{otherwise}
\end{cases}
\end{equation*}
The coloring function $c$ is a proper dominator coloring of $P_{n}$ which uses only $k+2$ colors, thereby completing the proof of this claim. $\Diamond$

Now that we have a parameterized value for the minimum dominator chromatic number over all orientations of all paths of length $4k+1$, we can use this to establish parameterized values for all other possible path lengths, thereby completing the proof of this theorem. Next we establish a relationship between $\chi_{d}(P_{4k+1})$ and $\chi_{d}(P_{4k+2})$.

\textbf{Claim 6.} $\chi_{d}(P_{4k+2})>\chi_{d}(P_{4k+1})$ for $k\geq 2$.

\textbf{Proof.} First, notice the exception in the statement of the theorem which states that $\chi_{d}(P_{6})=3$. If we let $P_{6}=v_{1}\dots v_{6}$ have out-degree sequence $\{1,0,2,0,2,0\}$, we may color vertices $v_{1}$, $v_{3}$, and $v_{5}$ with the same color since none of these vertices are dominated. The vertex $v_{2}$ mus be colored uniquely, as it is dominated by a vertex of out-degree one, but we may color $v_{4}$ and $v_{6}$ with the same color and obtain a proper dominator coloring of $P_{6}$ on $3$ colors. 

Since $P_{4k+1}\subset P_{4k+2}$, it follows that $\chi_{d}(P_{4k+1})\leq\chi_{d}(P_{4k+2})$. Moreover, any orientation of $P_{4k+2}$ which admits a minimum dominator coloring cannot have more than one vertex of out-degree one. To see this, choose any vertex of out-degree one and remove it, creating a path of length $4k+1$. If this path has any vertices of out-degree one, it does not admit a minimum dominator coloring over all paths of length $4k+1$. We obviously cannot use fewer colors by inserting the original vertex back into our path and re-obtaining our original path of length $4k+2$, so the original path of length $4k+2$ cannot have a dominator coloring on $\chi_{d}(P_{4k+1})$ colors.

If the out-degree sequence of $P_{4k+2}=v_{1}v_{2}\dots v_{4k+2}$ is given by the sequence $\{1,0,2,0,\dots,0,2,0\}$, the sub-path $v_{2}v_{3}\dots v_{4k+2}$ has length $4k+1$ and thus admits a unique coloring scheme. This dominator coloring cannot be extended to $P_{4k+2}$ since $v_{2}$ is not uniquely colored, therefore we may assume that the out-degree sequence of $P_{4k+2}$ is either $\{0,2,0,\dots 0,1,2,0,\dots 0,2,0\}$ or $\{0,2,0,\dots 0,2,1,0,\dots 0,2,0\}$. The proof for these two is similar, so, for concision, we only present the first case. Let $v_{i}$ be the vertex of out-degree one. Then $v_{i-1}$ and at least one of $\{v_{i},v_{i+2}\}$ must be uniquely colored. If $v_{i+2}$ is not uniquely colored, then we may remove $v_{i}$ to create a path of length $4k+1$ which has a proper dominator coloring. If this path uses fewest possible colors, we still obtain that $\chi_{d}(P_{4k+2})>\chi_{d}(P_{4k+1})$ since $v_{i}$ is uniquely colored. Now, if $v_{i+2}$ is uniquely colored, then $v_{i}$ may not be uniquely colored. However, by removing $v_{i}$, we obtain an orientation of $P_{4k+1}$ that does not admit a minimum dominator coloring (both out-neighbors of $v_{i+1}$ are uniquely colored), and so our dominator coloring of $P_{4k+2}$ uses strictly more than $\chi_{d}(P_{4k+1})$ colors.

Thus we may conclude that, for $k\geq2$, $\chi_{d}(P_{4K+2})=\chi_{d}(P_{4k+1})+1=k+3$, thus completing the proof of this claim. $\Diamond$

\textbf{Claim 7.} $\chi_{d}(P_{4k+3})=k+3$.

\textbf{Proof.} Since $\chi_{d}(P_{4k+1})=k+2$ and $\chi_{d}(P_{4k+2})=k+3$, it suffices to show that we can extend a minimum dominator coloring of $P_{4k+1}$ to $P_{4k+3}$ by adding only one new color. Let $P_{4k+1}=v_{1}v_{2}\dots v_{4k+1}$ be a subset of $P_{4k+3}=P_{4k+1}v_{4k+2}v_{4k+3}$, let the out-degree sequence of $P_{4k+3}=\{0,2,0,\dots,0,2,0\}$, and let $\mathcal{C}$ be a minimum dominator coloring of $P_{4k+1}$ with the structure provided in Claim 4. We complete the proof of this claim by coloring $v_{4k+2}$ with the color assigned to the non-dominated vertices of $P_{4k+1}$ and by assigning a new color to $v_{4k+3}$. $\Diamond$

\textbf{Claim 8.} $\chi_{d}(P_{4k+4})=\chi_{d}(P_{4k+3})$.

\textbf{Proof.} First, notice that this claim is equivalent to proving that $\chi_{d}(P_{4k})=\chi_{d}(P_{4k+1})$ since $P_{4k+4}$ is just $P_{4(k+1)}$ and the theorem claims that $\chi_{d}(P_{4(k+1)})=\chi_{d}(P_{4k+3})=\chi_{d}(P_{4k+1})+1$.

Our previous claim gave us a specific orientation and well-defined minimum possible dominator coloring for $P_{4k+3}$. Using this orientation and coloring, call it $\mathcal{C}$, of $P_{4k+3}$, we add a new vertex, $v_{4k+4}$, which dominates $v_{4k+3}$. Since $v_{4k+3}$ was the lone member of its color class in $\mathcal{C}$, $v_{4k+4}$ dominates a color class. To ensure that the dominator coloring of $P_{4k+4}$ is minimum and proper and thereby complete the proof of this claim, we assign the vertex $v_{4k+4}$ to the same color class that contains all out-degree two vertices. $\Diamond$

This completes the proof of the theorem.
\end{proof}

Thus we have not only bounded in both directions the dominator chromatic number of all possible orientations of a path, but also characterized specific orientations which attain these bounds. 

We conclude this section by showing a very interesting result relating the dominator chromatic number of digraphs to digraph parameters. To motivate this result, recall that the chromatic number of graphs satisfies the relation $\lim\limits_{n\to\infty}\frac{\chi(G)}{\Delta(G)}\leq\lim\limits_{n\to\infty}\frac{\Delta(G)+1}{\Delta(G)}=1$.

However, the dominator chromatic number does not obey this limit in general. Using the results we just obtained on the minimum dominator chromatic number of orientations of paths, we can actually obtain that $\lim\limits_{n\to\infty}\frac{\chi_{d}(D)}{\Delta(D)}$ actually can go to infinity in certain cases, such as that of orientations of paths. As we will see later, if we choose a different family of graphs, such as tournaments, the original limit on the chromatic number of graphs and this limit on the dominator chromatic number of digraphs are equal. Phrased slightly differently, this becomes $\limsup\limits_{n\to\infty}\frac{\chi_{d}(D)}{\Delta(D)}\to\infty$.

An interesting problem, one which will be recounted in the Conclusion, would be to describe families of digraphs which have a finite $\limsup\limits_{n\to\infty}\frac{\chi_{d}(D)}{\Delta(D)}=r$ for $r\in\mathbb{R}$.

\section{Orientations of Cycles}\label{s4}
We next turn our attention to orientations of cycles. To do this, we first prove a lemma that considers the problem of embedding paths into cycles.
\begin{lemma}\label{l3}
$\chi_{d}(P_{m})\leq\chi_{d}(C_{m})$ for $m\neq4$.
\end{lemma}
\begin{proof}
First, see that a minimum dominator coloring orientation of $P_{m}$ can be embedded into some orientations of $C_{m}$. Since the arc $C_{m}\setminus P_{m}$ can combine at most two color classes, we need only consider the case where $P_{m}$ admits a minimum dominator coloring, for if a cycle $C_{m}$ contains no minimum dominator coloring of any subpath $P_{m}$, $\chi_{d}(C_{m})\geq\chi_{d}(P_{m})$. Thus, to prove this lemma, we will show that any orientation of $C_{m}$ which contains a minimum orientation of $P_{m}$ cannot have a smaller dominator chromatic number, with the exception of $C_{4}$ which can indeed combine two color classes of a minimum dominator coloring of $P_{4}$ in precisely this manner.

Let $P_{m}=v_{1}v_{2}\dots v_{m}$ be a subpath of $C_{m}$ which admits a minimum dominator coloring over all orientations of $P_{m}$. In order for $\chi_{d}(C_{m})<\chi_{d}(P_{m})$, the arc between $v_{1}$ ad $v_{m}$ allows for two color classes to be merged. From this fact, we may conclude that, without loss of generality, $v_{1}$ and $v_{m-1}$ are each uniquely colored in $P_{m}$ (notice that the vertex $v_{m}$ must have out-degree equal to one, for if both end vertices of $P_{m}$ are uniquely colored, the added arc $v_{m}v_{1}$ clearly cannot combine color classes). 

First consider the case where $m=4k+1$ for some $k\in\mathbb{N}$. Since $P_{m}$ has a unique structure and coloring scheme which admits a minimum dominator coloring, it follows that $\chi_{d}(C_{m})\geq\chi_{d}(P_{m})$ since $m-1=4k$ and $1\not\equiv 4k\ (\mathrm{mod}\ 4)$.

Next, consider the case where $m=4k+2$ for some $k\in\mathbb{N}$. Since we may assume that $v_{m}$ has out-degree one, and since there is exactly one vertex of out-degree one in a minimum dominator coloring of $P_{4k+2}$, we may infer that the out-degree sequence of $P_{m}$ is $\{0,2,0,2,\dots,2,0,1\}$, that the color scheme of the vertices $v_{1}$ through $v_{m-1}$ is identical to that of the first $4k$ vertices of a minimum dominator coloring of the path $P_{4k+1}$ ($m-1=4k+1$ in this case), that we may assign $v_{m}$ to the same color class as the vertices of out-degree two, and that the vertex $v_{m-1}$ is uniquely colored. This implies that $v_{1}$ is not uniquely colored, but since $v_{m-1}$ is uniquely colored, we may recolor the odd-indexed vertices by the following color scheme that is a variant of the color scheme of a minimum dominator coloring of the path $P_{4k+1}$ and obtain a minimum dominator coloring of $P_{4k+2}$ that assigns $v_{1}$ to a unique color class.
\begin{equation*}
c(v_{i})=\begin{cases}
c_{1} & \mathrm{if}\ i\equiv\ 0\ (\mathrm{mod}\ 2)\\
c_{2} & \mathrm{if}\ i\equiv\ 3\ (\mathrm{mod}\ 4)\\
\mathrm{unique} & \mathrm{otherwise}
\end{cases}
\end{equation*}
However, under this color scheme, the vertices $v_{1}$ and $v_{m-1}$ cannot belong to the same color class in $C_{m}$ for if they do, the vertex $v_{2}$ no longer dominates any color class. It is east to see that any other dominator coloring will require more colors.

For our next case, we consider $m=4k$ for $k>1$. To show that $\chi_{d}(P_{4k})\leq\chi_{d}(C_{4k})$, consider a minimum dominator coloring of $C_{4k}$. If there exists some vertex $v\in V(C_{4k})$ such that $d^{+}(v)=0$ and $v$ is not uniquely colored, then we may split $v$ into two non-adjacent vertices, each of the same color as $v$, creating a proper dominator coloring of an orientation of $P_{4k+1}$. Since $\chi_{d}(P_{4k})=\chi_{d}(P_{4k+1})$, this implies that $\chi_{d}(P_{4k})\leq\chi_{d}(C_{4k})$. Thus we need only to show that there exists such a vertex in any minimum dominator coloring of $C_{4k}$.

To accomplish this, we first show that there are no consecutive vertices of out-degree one in a minimum orientation. Assume there are. Then there a subset $\{v_{i},v_{i+1},v_{i+2},v_{i+3}\}$ of $V(C_{4k})$ with out-degree sequence either $\{2,1,1,0\}$ or $\{1,1,1,0\}$. In either case we may reverse the orientation of the arc $v_{i+1}v_{i+2}$ to $v_{i+2}v_{i+1}$ and change only the color of the vertex $v_{i+2}$ to the color assigned to the vertices of out-degree two (such a vertex exists because there is a bijection between out-degree zero and out-degree two vertices in an orientation of a cycle). This contradicts the assumption that the dominator coloring was minimum.

Next we show that there are no vertices of out-degree one. Since there are no consecutive vertices of out-degree one, the existence of a vertex of out-degree one implies that there is a subsequence $\{v_{i},v_{i+1},v_{i+2}\}$ of $C_{4k}=v_{1}v_{2}\dots v_{4k}v_{1}$ with the out-degree sequence $\{2,1,0\}$. If $v_{i}$ is uniquely colored, then we may remove $v_{i}$ and insert it after some other vertex of out-degree one, maintaining a proper dominator coloring of $C_{4k}$ using the same number of colors. But we just showed that any dominator coloring of $C_{4k}$ with consecutive vertices of out-degree one is not a minimum dominator coloring. Thus we can assume that no vertex of out-degree one is uniquely colored. By removing all vertices of out-degree one, say $2m$, we reduce the number of colors used by some value $r\geq1$. Additionally, each vertex of out-degree two that was adjacent to a vertex of out-degree one now dominates two color classes comprised of one element each. Since we can recolor one of these two dominated vertices with the color of the other dominated vertex in every case, at worst every two vertices of out-degree one correspond to a reduction in the number of colors used in this smaller cycle (it is possible that a vertex of out-degree two dominated two vertices of out-degree one that together comprised an entire color class in the original dominator coloring of $C_{4k}$). Thus we have reduced the number of colors used by at least $m+1$, and so it suffices to show that we can insert these $2m$ vertices using no more than $m$ colors. To do this, simply chose a vertex of out-degree two, call it $v_{i}$, which dominates a single color class of two vertices, i.e., a vertex of out-degree two which was adjacent to a vertex of out-degree one in the original orientation of $C_{4k}$. Between $V_{i}$ and $v_{i+1}$ insert the removed vertices ($\{x_{1},\dots,x_{2m}\}$) by orienting them such that their out-degree is $\{0,2,\dots,0,2\}$. Since the vertices now having out-degree two can all be assigned to an existing color class, this leaves only $m$ vertices needing color assignments. If we color $x_{1}$ with the same color as $v_{i-1}$ and uncolor the vertex $v_{i+1}$, we still have only $m$ vertices needing colors assigned to them. Even if each receives a unique color, we have still reduced the number of colors used to properly dominator color $C_{4k}$. Thus a minimum dominator coloring of $C_{4k}$ has no vertices of out-degree one.

From here it is obvious that any minimum dominator coloring o $C_{4k}$ must have non-uniquely colored vertices of out-degree zero.

Lastly, consider the case when $m=4k+3$. Particularly, consider a minimum dominator coloring of $C_{4k+3}$. Since Theorem \ref{t1} tells us that $\chi_{d}(P_{4k+3})=\chi_{d}(P_{4(k+1)})$, and since we just established that $\chi_{d}(P_{4(k+1)})\leq\chi_{d}(C_{4(k+1)})$, it suffices to prove that $\chi_{d}(C_{4(k+1)})\leq\chi_{d}(C_{4k+3})$. To do this, we simply extend our minimum dominator coloring of $C_{4k+3}$ to a proper dominator coloring of $C_{4(k+1)}$.

With a similar argument to that used to establish valid out-degree orientations for minimum dominator colorings in the case of $C_{4k}$, we can assume the existence of a subsequence $\{v_{i-1},v_{i},v_{i+1}\}$ of $V(C_{4k+3})$ which has out-degree sequence $\{2,1,0\}$. If $v_{i}$ is uniquely colored, then we may insert a vertex $u$ between $v_{i-1}$ and $v_{i}$ that has out-degree zero, creating an orientation of $C_{4(k+1)}$. By coloring $u$ with $c(v_{i})$ and by recoloring $v_{i}$ with $c(v_{i-1})$, we create a proper dominator coloring of $C_{4(k+1)}$ on $\chi_{d}(C_{4k+3})$ colors, establishing that $\chi_{d}(C_{4k+3})\geq\chi_{d}(C_{4(k+1)})$ as desired.

Next assume that the vertex $v_{i}$ is not uniquely colored. Either the vertex $v_{i-2}$ and $v_{i}$ together comprise an entire color class, or the vertex $v_{i-2}$ is uniquely colored, else $v_{i-1}$ does not dominate any color class, contradicting our assumption of a minimum dominator coloring of $C_{4k+3}$. If $c(v_{i-2})=c(v_{i})$, then we may insert a vertex $u$ of out-degree zero between $v_{i-1}$ and $v_{i}$, color $u$ with $c(v_{i})$, and recolor $v_{i}$ with $c(v_{i-1})$. If $v_{i-2}$ is uniquely colored, we may insert a vertex $u$ of out-degree zero between $v_{i-1}$ and $v_{i}$, color $u$ with $c(v_{i})$, and recolor $v_{i}$ with $c(v_{i-1})$. In either case we again establish the inequality $\chi_{d}(C_{4k+3})\geq\chi_{d}(C_{4(k+1)})$. This completes the case of $m=4k+3$ and thus completes the proof of the lemma.
\end{proof}

With this very important lemma proven, we are now ready to prove the minimum dominator chromatic number of cycles.
\begin{theorem}\label{t2}
The minimum dominator chromatic number over all orientations of the cycle $C_{n}$ is given by $\chi_{d}(C_{n})=k+2$ where $n=4k-i$ for $i\in\{0,1,2,3\}$\footnote{Do notice the definition of $n$ as $n=4k-i$ for $i\in\{0,1,2,3\}$. While being slightly cumbersome notationally, this expression most succinctly expresses the value of the smallest possible dominator chromatic number as a function of the parameter $k$.} with the exceptions $\chi_{d}(C_{4})=2$ and $\chi_{d}(C_{5})=\chi_{d}(C_{6})=3$.
\end{theorem}
\begin{proof}
In the same fashion as the last proof, this proof will consist of a series of claims and proofs of these claims which, in total, will prove this theorem. Observe first that the cycle $C_{3}$ is a tournament and thus $\chi_{d}(C_{3})=3$, and also that $\chi_{d}(C_{5})=\chi_{d}(C_{6})=3$.

\textbf{Claim 1.} Let $S=\{v\in V(C_{n})|d^{+}(v)=2\}$. Then $S$ is colored completely by one color class in any minimum dominator coloring of $C_{n}$.

\textbf{Proof.} No vertex in any orientation of a cycle can dominate a vertex with out-degree two, hence all out-degree two vertices may be assigned to the same color class. $\Diamond$

\textbf{Claim 2.} Let $C_{n}=v_{1}v_{2}\dots v_{n}v_{1}$ be an orientation of a cycle on $n$ vertices and let $\mathcal{C}$ be a dominator coloring of $C_{n}$ using fewest possible colors. Then we have that $\nexists\ v_{i},v_{i+1}\in V(C_{n})$ s.t. $d^{+}(v_{i})=d^{+}(v_{i+1})=1$.

\textbf{Proof.} Assume not. First, consider the case in which the out-degree sequence for $C_{n}$ is $\{1,1,1,\dots,1,1,1\}$. It is immediately obvious that the out-degree sequence $\{0,2,0,2,\dots,0,2\}$ is an orientation of $C_{n}$ which can be dominator colored with fewer colors than were used when the out-degree sequence was $\{1,1,1,\dots,1,1,1\}$. This implies that $\exists\ u,v\in V(C_{n})$ such that $d^{+}(u)=0$ and $d^{+}(v)=2$.

Next, assume that there is a sequence of at least two consecutive vertices in $C_{n}$ which all have out-degree equal to one. Do notice that these vertices must be preceded by a vertex of out-degree two. Let the subsequence of $V(C_{n})$ whose out-degree sequence is $\{2,1,1\}$ be denoted by $\{v_{i},v_{i+1},v_{i+2}\}$. We know that the vertex $v_{i+2}$ and the out-neighbor of $v_{i+2}$ are both uniquely colored since each of these vertices is dominated by a vertex of out-degree one. If we reverse the orientation of the arc $v_{i+1}v_{i+2}$ so that it becomes $v_{i+2}v_{i+1}$, we may recolor the vertex $v_{i+2}$ with the same color that was assigned to $v_{i}$ and all other vertices with out-degree two since we now have $d^{+}(v_{i+2})=2$. This alone suffices to complete the proof since $v_{i+2}$ was previously uniquely colored. $\Diamond$

\textbf{Claim 3.} There are no vertices with out-degree one in any minimum dominator coloring of the cycle $C_{4k}$.

\textbf{Proof.} This has been established in the proof of Lemma \ref{l3} when showing that $\chi_{d}(P_{4k})\leq\chi_{d}(C_{4k})$. $\Diamond$

With these first three claims intact, we are now ready to begin to prove results on the smallest possible dominator chromatic number for cycles.

\textbf{Claim 4.} $\chi_{d}(C_{4k})=k+2$.

\textbf{Proof.} Let $\mathcal{C}$ be a minimum dominator coloring of the cycle $C_{4k}$. By choosing any non-uniquely colored vertex of out-degree zero in $C_{4k}$, call it $v_{1}$ for convenience, we may split $v_{1}$ into two non-adjacent vertices $v_{1}$ and $v_{4k+1}$, each with the same color, thus creating a path of length $4k+1$. Since all vertices of out-degree two still dominate a color class, this constitutes a proper dominator coloring of an orientation of the path $P_{4k+1}$, hence $\chi_{d}(P_{4k+1})\leq\chi_{d}(C_{4k})$.

Next, let $\mathcal{C}$ be a minimum proper dominator coloring of the path $P_{4k+1}=v_{1}v_{2}\dots v_{4k+1}$. From Theorem \ref{t1} we know that both end vertices of $P_{4k+1}$ are colored with the same color. Given this, we may merge the vertices $v_{1}$ and $v_{4k+1}$ into a single vertex with the same color as $v_{1}$ and $v_{4k+1}$ to create a cycle of length $4k$ that has a proper dominator coloring, whence $\chi_{d}(C_{4k})\leq\chi_{d}(P_{4k+1})$ and the proof of this claim is complete. $\Diamond$

\textbf{Claim 5.} $\chi_{d}(C_{4(k+2)})=\chi_{d}(C_{4(k+1)})+1$ for $k\in\mathbb{N}$.

\textbf{Proof.} First observe that $\chi_{d}(C_{4})=2$ and $\chi_{d}(C_{8})=4$, hence the importance of the indexing in this claim. From Claim 3 we know that the out-degree sequence of any cycle is precisely $\{0,2,0,2,\dots,0,2,0,2\}$. Let $v_{i}\in V(C_{4(k+1)})$ have out-degree two and let $v_{i-1}$ be uniquely colored. By inserting four consecutive vertices in between $v_{i-1}$ and $v_{i}$ with out-degree sequence $\{2,0,2,0\}$ we may extend this dominator coloring to $C_{4(k+2)}$ by using only one more color as follows. Call these four new vertices $w$, $x$, $y$, and $z$. We may color $w$ and $y$ with the same color as $v_{i}$ since they all have out-degree two. The vertex $x$ may be colored with a non-dominated color class for vertices of out-degree zero. By coloring the vertex $z$ uniquely, we show that $\chi_{d}(C_{4(k+2)})\leq\chi_{d}(C_{4(k+1)})+1$ ($\chi_{d}(C_{4(k+2)})$ cannot be less than $\chi_{d}(C_{4(k+1)})$ due to Theorem \ref{t1} and Lemma \ref{l3}).

Now assume that there is a smallest counterexample to this claim, call it $C_{4(k+2)}$ for some fixed $k>1$. It must be the case that $\chi_{d}(C_{4(k+2)})=\chi_{d}(C_{4(k+1)})$. But if this is true, we may remove any four consecutive vertices from $C_{4(k+2)}$ and get a proper dominator coloring of $C_{4(k+1)}$ on fewer colors than $\chi_{d}(C_{4(k+2)})$, a contradiction, To see that this holds, see that any four consecutive vertices, say $v_{i}$ through $v_{i+3}$, of $C_{4(k+2)}$ have out-degree sequence $\{0,2,0,2\}$ or $\{2,0,2,0\}$. Since the entire out-neighborhood of $v_{i+1}$ ($v_{i+2}$, respectively) is included in this sequence, thus some entire color class is also contained in this sequence. This allows us to derive the desired contradiction and completes the proof of this claim. $\Diamond$

\textbf{Claim 6.} $\chi_{d}(C_{4k+1})>\chi_{d}(C_{4k})$.

\textbf{Proof.} From Theorem \ref{t1}, Lemma \ref{l3}, and Claim 4 of this theorem, we know that $\chi_{d}(C_{4k})=\chi_{d}(P_{4k})=\chi_{d}(P_{4k+1})\leq\chi_{d}(C_{4k+1})$. Hence it suffices to show that $\chi_{d}(P_{4k+1})<\chi_{d}(C_{4k+1})$. Since we know that there is a unique orientation and color scheme combination for $P_{4k+1}$ which admits a minimum dominator coloring over all orientations of $P_{4k+1}$, and since the end vertices of this $P_{4k+1}$ are given the same color, it is impossible to extend a minimum dominator coloring of $P_{4k+1}$ to $C_{4k+1}$ using only the same color pallet. Hence we may assume that any sub-path $P_{4k+1}$ of any orientation of $C_{4k+1}$ which attains a minimum dominator coloring over all orientations of $C_{4k+1}$ must use at least $k+3$ colors rather than the $k+2$ colors specified in Theorem \ref{t1}. Let $P_{4k+1}=v_{1}v_{2}\dots v_{4k+1}$ and, without loss of generality, assume that we are adding the arc $v_{4k+1}v_{1}$ to complete $C_{4k+1}$. Then it must be the case that $v_{1}$ and $v_{4k}$ were both uniquely colored, else we cannot reduce the number of colors used in our dominator coloring of $C_{4k+1}$. Consider the vertex $v_{1}$. Either $d^{+}(v_{1})=1$ or $d^{+}(v_{1})=0$. 

If $d^{+}(v_{1})=1$, then $v_{2}$ must also be uniquely colored. This means that the vertices $S=\{v_{1},v_{2},v_{4k},v_{4k+1}\}$ collective require four three colors (each of $v_{1}$, $v_{2}$, and $v_{4k}$ are uniquely colored). Of these four colors, only the color assigned to $v_{4k+1}$ may appear elsewhere in $C_{4k+1}$, else we do not have a proper dominator coloring. Since the induced subgraph $D[V(D)\setminus S]$ amounts to an orientation of $P_{4(k-1)+1}$, we know that we must use at least $(k-1)+2=k+1$ colors to color these vertices. Together this all implies that we need at least $k+1+2=k+3$ colors for $C_{4k+1}$ if $d^{+}(v_{1})=1$, even if we combine a color class when creating $C_{4k+1}$ from $P_{4k+1}$.

If $d^{+}(v_{1})=0$ then we may color the vertices of $S$ with three colors (one each for $v_{1}$ and $v_{4k}$, but $v_{2}$ and $v_{4k+1}$ may belong to the same color class). Again see that the induced subgraph $D[V(D)\setminus S]$ is an orientation of $P_{4(k-1)+1}$ and requires at least $k+1$ colors in any proper dominator coloring. Moreover, this subpath can only attain $k+1$ colors if both end vertices are not uniquely colored. Since we assumed $v_{1}$ is uniquely colored, we may add $v_{1}$ and $v_{2}$ back to the path, creating a path of length $4(k-1)+3$ using $k+2$ colors. Adding $v_{4k}$ and $v_{4k+1}$ using their original colors implies that the vertex $v_{4k-1}$, the vertex that was the end vertex of the path $D[V(D)\setminus S]$ of length $4(k-1)+1$ has out-degree one in $P_{4k+1}$ and thus must dominate the color class assigned to $v_{4k}$. This implies that if we use $k+3$ colors to properly dominator color $P_{4k+1}$ and require the vertices $v_{1}$ and $v_{4k}$ to each be uniquely colored, adding the arc $v_{4k+1}v_{1}$ does not permit the vertex $v_{4k+1}$ to combine color classes in $C_{4k+1}$. Hence $\chi_{d}(C_{4k+1})\geq k+3$. Since this now covers all possible cases, the proof of this claim is complete. $\Diamond$

\textbf{Claim 7.} $\chi_{d}(C_{4k})>\chi_{d}(C_{4k-i})$ for $i\in\{1,2,3\}$.

\textbf{Proof.} First, see that Claims 5 and 6 just established this result for the case of $i=3$, and the case of $i=1$ was actually proven in the conclusion of the proof of Lemma \ref{l3} (in the case of $m=4k+3$). Thus we need only to prove that this holds in the case of $i=2$. 

Our goal is to construct a minimum dominator coloring of $C_{4k+2}$ which can be extended to a dominator coloring of $C_{4(k+1)}$ without the addition of new color classes since Claims 4 and 5 have established for us that $\chi_{d}(C_{4k})=k+2$ and $\chi_{d}(C_{4(k+1)})=k+3$, and since Theorem \ref{t1} and Lemma \ref{l3} combine to tells us that $k+3=\chi_{d}(P_{4k+2})\leq\chi_{d}(C_{4k+2})$. To do this, begin with a minimum dominator coloring of the smaller path $P_{4k+1}=v_{1}v_{2}\dots v_{4k+1}$. We know precisely what this looks like, and that it uses $k+2$ colors, so we may create our minimum dominator coloring of $P_{4k+2}$ by recoloring the vertex $v_{4k+1}$ with a new color, adding the vertex $v_{4k+2}$, assigning $v_{4k+2}$ to the same color class as $v_{4k}$ (which is a vertex of out-degree two in $P_{4k+1}$), and by adding the arc $v_{4k+2}v_{4k+1}$. Since this uses $k+3$ colors, this is a minimum dominator coloring of $P_{4k+2}$. Since $\chi_{d}(P_{4k+2})\leq\chi_{d}(C_{4k+2})$, the arc $v_{4k+2}v_{1}$ establishes a minimum dominator coloring of $C_{4k+2}$ on precisely $k+3$ colors. This establishes that $\chi_{d}(C_{4k})<\chi_{d}(C_{4k+2})=\chi_{d}()$, and the proof of this last case is complete. $\Diamond$

This completes the proof of the theorem.
\end{proof}
\begin{corollary}\label{cor1}
The directed cycle is the unique orientation of a cycle which maximizes the dominator chromatic number of an orientation of a cycle.
\end{corollary}

\section{Orientations of $K_{n}$ and $K_{m,n}$}\label{s5}
Orientations of complete graphs, i.e., tournaments, are rather easily characterized. As a direct consequence of Observations \ref{o1} and \ref{o2}, we provide the dominator chromatic number of any orientation of a complete graph with the following observation.
\begin{obs}\label{o3}
For any tournament $T_{n}$ on $n$ vertices, we have $\chi_{d}(T_{n})=n$.
\end{obs}

As it turns out, complete bipartite graphs are also a very important class of digraphs when it comes to dominator colorings. The following theorem provides another complete characterization of a very important problem in dominator colorings of digraphs.
\begin{theorem}\label{t3}
Let $D$ be a simple, connected digraph. Then $\chi_{d}(D)=2\iff D=K_{m,n}$ with partite sets $X$ and $Y$ satisfying $xy\in A(D)\ \forall\ x\in X\ \mathrm{and}\ \forall\ y\in Y$.
\end{theorem}
\begin{proof}
($\impliedby$) Let $V(D)=\{X,Y\}$ be a bipartition of $D=K_{m,n}$ satisfying $xy\in A(D)\ \forall\ x\in X\ \mathrm{and}\ \forall\ y\in Y$. Color $x$ with $c_{1}$ for all $x\in X$ and color $y$ with $c_{2}$ for all $y\in Y$.

($\implies$) Let $D$ be a simple, connected digraph with $\chi_{d}(D)=2$. Partition $V(D)$ into two sets $X$ and $Y$ such that $X=\{v\in V(D)|c(v)=c_{1}\}$ and $Y=\{v\in V(D)|c(v)=c_{2}\}$. Since this is a proper dominator coloring of $D$, there do not exists arcs of the form $x_{i}x_{j}$ for $x_{i},x_{j}\in X$ or of the form $y_{i}y_{j}$ for $y_{i},y_{j}\in Y$, hence $\{X,Y\}$ is a bipartition of $V(D)$. Without loss of generality, assume that there exists the arc $x^{\star}y^{\star}\in A(D)$ for $x^{\star}\in X$ and $y^{\star}\in Y$. Since $c(y)=c_{2}\ \forall\ y\in Y$, and since this is a proper dominator coloring of $D$, it follows that $x^{\star}y\in A(D)\ \forall\ y\in Y$. Since $D$ is connected, we have that $\forall\ x\in X\ \exists\ y\in Y$ such that either $xy\in A(D)$ or $yx\in A(D)$. Assume that there exists some $x\in X$ and $y\in Y$ such that $yx\in A(D)$. We know already that $x^{\star}y\in A(D)$ and $c(x)=c(c^{\star})$, so $y$ does not dominate any color class in $D$, a contradiction. Thus $d^{+}(y)=0\ \forall\ y\in Y$ and all that remains to be shown is that $D$ is complete. Assume that $\exists\ \hat{x}\in X$ and $\exists\ \hat{y}\in Y$ such that $\hat{x}\hat{y}\not\in A(D)$. Since $D$ is connected and since $d^{+}(y)=0\ \forall\ y\in Y$, it follows that $d^{+}(\hat{x})\geq 1$. Thus, for $D$ to have a proper dominator coloring, it must be that $\hat{x}\hat{y}\in A(D)$. Therefore, if $\chi_{d}(D)=2$ for a simple, connected digraph $D$, it must be the case that $D=K_{m,n}$.
\end{proof}

\section{Quantifying the Effect of Orientation on Dominator Colorings}\label{s6}
In this section we introduce an interesting graph invariant, $\varsigma^{\star}(D)$, which is the difference between the minimum dominator chromatic number over all all orientations of the digraph and the chromatic number of the underlying graph. This invariant tell us about how impactful orientations of the underlying graph structure are on vertex coloring and domination problems. Formally, the graph invariant $\varsigma^{\star}(D)$ is defined by the following two equations.
\begin{equation}
\varsigma(D)=\chi_{d}(D)-\chi(G_{D})
\end{equation}
\begin{equation}
\varsigma^{\star}(D)=\max\{\varsigma(D)\}\ \mathrm{over\ all\ orientations\ of\ G_{D}}
\end{equation}

With this formal definition intact, we provided several initial results on this graph invariant.
\begin{result}\label{r1}
For all orientations of $K_{n}$, $\varsigma(K_{n})=\varsigma^{\star}(K_{n})=0$.
\end{result}
\begin{proof}
This follows directly from Observation \ref{o3}.
\end{proof}
\begin{result}\label{r2}
For the orientation of  $K_{m,n}$ where all arcs orient from $X$ to $Y$, $\varsigma(K_{m,n})=\varsigma^{\star}(K_{m,n})==0$.
\end{result}
\begin{proof}
This follows directly from Theorem \ref{t3}.
\end{proof}
\begin{result}\label{r3}
For orientations of the path $P_{n}$, $\varsigma(P_{n})=0\iff P_{n}$ is the directed path on $n$ vertices.
\end{result}
\begin{proof}
This follows directly from Theorem \ref{t1}.
\end{proof}
\begin{result}\label{r4}
For orientations of the path $C_{n}$, $\varsigma(C_{n})=0\iff C_{n}$ is the directed cycle on $n$ vertices.
\end{result}
\begin{proof}
This follows directly from Theorem \ref{t2}.
\end{proof}
\begin{result}\label{r5}
For orientations of the path $P_{n}$, we have
\begin{equation*}
\varsigma^{\star}(P_{n})=\begin{cases}
3k-2\ \mathrm{if}\ n=4k\\
3k-1\ \mathrm{if}\ n=4k+1\\
3k-1\ \mathrm{if}\ n=4k=2\\
3k\quad\quad \mathrm{if}\ n=4k=3
\end{cases}
\end{equation*}
\end{result}
\begin{proof}
This follows directly from Theorem \ref{t1}.
\end{proof}
\begin{result}\label{r6}
For orientations of $C_{n}$ we have
\begin{equation*}
\varsigma^{\star}(C_{n})=\begin{cases}
3k-2\ \mathrm{if}\ n=4k\\
3k-2\ \mathrm{if}\ n=4k+1\\
3k-1\ \mathrm{if}\ n=4k=2\\
3k\quad\quad \mathrm{if}\ n=4k=3
\end{cases}
\end{equation*}
\end{result}
\begin{proof}
This follows directly from Theorem \ref{t2}.
\end{proof}

We conclude this section by mentioning that Observation \ref{o1} implies that the digraph invariant is necessarily a non-negative integer.

\section{Conclusion}\label{s7}
This paper developed the beginnings of a theory of dominator coloring for directed graphs. As it turns out, orienting arcs makes the notion of domination much more complicated, so rather than attempting to describe dominator chromatic numbers of digraphs for a given orientation, we focused on proving the minimum dominator chromatic number over all possible orientations for a given graph structure. In particular we proved the minimum dominator chromatic number for orientations of paths and cycles. We described all graphs of dominator chromatic number two with Theorem \ref{t3}.

Perhaps most notably, dominator coloring of digraph are particularly interesting for the fact that it is not generally the case that $\chi_{d}(H)\leq\chi_{d}(D)$ for a subgraph $H$ of a digraph $D$. The example provided, and only known, is that $\chi_{d}(P_{4})=3>2=\chi_{d}(C_{4})$ (notice that the orientations of each which attain minimum dominator colorings do indeed directly satisfy this property). This is a major deviation from what is expected from previous results in graph coloring, demonstrating the mathematical intrigue of dominator colorings of digraphs. Adding to this intrigue, we also saw that the dominator chromatic number of a digraph is not bounded by the maximum degree of the digraph. To aid in the future study of this deviance, we introduced the graph invariant $\varsigma^{\star}(D)$ which tells us how much orienting a graph can affect results in vertex coloring and domination.

To conclude this paper, we recognize some of the many potential avenues for further exploration into the topic of dominator colorings of digraphs.
\begin{problem}
Which digraphs satisfy $\chi_{d}(D)=\chi(D)$?
\end{problem}
Notice that by answering this problem we are answering an important question about the graph invariant $\varsigma^{\star}(D)$, namely which graphs admit an orientation such that $\varsigma^{\star}(D)=0$.
\begin{problem}
When is $\chi_{d}(D)$ invariant under all possible orientations of the underlying graph $G_{D}$?
\end{problem}
Notice that this problem may be restricted in the possibly more convenient manner.
\begin{problem}
When is $\chi_{d}(D)$ invariant under arc reversal? I.e., when does $\chi_{d}(D)=\chi_{d}(D^{-})$ where $D^{-}$ is the digraph on the same underlying graph $G_{D}$ as $D$, but the arc set $A(D^{-})=\{uv|vu\in A(D)\}$?
\end{problem}
\begin{problem}
How far can Theorem \ref{t3} be generalized in terms of domination among partite sets?
\end{problem}
Alternatively, we can attempt to generalize Theorem \ref{t3} in the following manner.
\begin{problem}
Which digraphs have dominator chromatic number $n$.
\end{problem}
\begin{problem}
Are there families of digraphs which satisfy $\limsup\limits_{n\to\infty}\frac{\chi_{d}(D)}{\Delta(D)}=r$ for some $r\in\mathbb{R}$? What are they? And if so, is this phenomenon related to $\varsigma^{\star}(D)$?
\end{problem}
Notice that the above problem may be defined for any of $\Delta^{+}(D)$, $\Delta^{-}(D)$, or $\Delta(D)=\Delta(G_{D})$.
\begin{problem}
How does $\varsigma^{\star}(D)$ behave with respect to graph operations? What about subgraph containment?
\end{problem}
\begin{problem}
Which families of graphs and subgraphs admit a positive dominator discrepancy $\delta(D,H)$?
\end{problem}

\bibliography{DCD1}
\end{document}